\documentclass[twoside,11pt]{amsart} \usepackage{epsfig}
\usepackage{latexsym} 
\usepackage{amsmath} \usepackage{amssymb}



\newtheorem{theorem}{Theorem}[section]
\newtheorem{lemma}[theorem]{Lemma}
\newtheorem{corollary}[theorem]{Corollary}
\newtheorem{remark}[theorem]{Remark}

\newdimen\epsfxsize
\newdimen\epsfysize

\newcommand {\gap}     {\makebox[0.075 in]{}}   
\newcommand {\st}      {\gap : \gap}   

\newcommand {\fto}     {\longrightarrow}

\newcommand {\set}[1]  {\left\{ {#1} \right\}}   
\newcommand {\ord}[1]  {{#1}^{\rm th}}

\newcommand{\primeprod}[1] {{#1}^\#}

\newcommand{\Z}     {{\mathbb Z}}

\newcommand{\Rat}[2]  {w_{{#2},{#1}}}  

\newcommand{\pgap}   {{\mathcal G}}

\newcommand{\lil}{\scriptstyle}

\begin{document}

\title{On Polignac's conjecture }

\date{15 Feb 2014}

\author{Fred B. Holt and Helgi Rudd}
\address{fbholt@uw.edu ;  4311-11th Ave NE \#500, Seattle, WA 98105;
48B York Place, Prahran, Australia 3181}

\begin{abstract}
A few years ago we identified a recursion that works directly with the gaps among
the generators in each stage of Eratosthenes sieve.  This recursion provides explicit
enumerations of sequences of gaps among the generators, which are known as 
constellations.

As the recursion proceeds, adjacent gaps within longer constellations are added together
to produce shorter constellations of the same sum.  These additions or closures correspond to
removing composite numbers that are divisible by the prime for that stage of Eratosthenes sieve.
Although we don't know where in the cycle of gaps a closure will occur, we can enumerate 
exactly how many copies of various constellations will survive each stage.

In this paper, we broaden our study of these systems of constellations of a fixed sum.  
By generalizing our methods, we are able to demonstrate that for every even number $2n$
the gap $g=2n$ occurs infinitely often through the stages of Eratosthenes sieve.
Moreover, we show that asymptotically the ratio of the number of 
gaps $g=2n$ to the number of gaps $g=2$ at each stage of Eratosthenes sieve
converges to the estimates made for gaps among primes by Hardy and Littlewood
in Conjecture B of their 1923 paper.

\end{abstract}

\maketitle

\section{Introduction}
We work with the prime numbers in ascending order, denoting the
$\ord{k}$ prime by $p_k$.  Accompanying the sequence of primes
is the sequence of gaps between consecutive primes.
We denote the gap between $p_k$ and $p_{k+1}$ by
$g_k=p_{k+1}-p_k.$
These sequences begin
$$
\begin{array}{rrrrrrc}
p_1=2, & p_2=3, & p_3=5, & p_4=7, & p_5=11, & p_6=13, & \ldots\\
g_1=1, & g_2=2, & g_3=2, & g_4=4, & g_5=2, & g_6=4, & \ldots
\end{array}
$$

A number $d$ is the {\em difference} between prime numbers if there are
two prime numbers, $p$ and $q$, such that $q-p=d$.  There are already
many interesting results and open questions about differences between
prime numbers; a seminal and inspirational work about differences
between primes is Hardy and Littlewood's 1923 paper \cite{HL}.

A number $g$ is a {\em gap} between prime numbers if it is the difference
between consecutive primes; that is, $p=p_i$ and $q=p_{i+1}$ and
$q-p=g$.
Differences of $2$ or $4$ are also gaps; so open questions
like the Twin Prime Conjecture, that there are an infinite number
of gaps $g_k=2$, can be formulated as questions about differences
as well.

{\em Polignac's conjecture.}  In 1849 de Polignac conjectured that 
for every $n >0$ the gap $g=2n$ occurs infinitely often among primes.

In this paper we provide supporting evidence for this conjecture by proving
that the analogue for Eratosthenes sieve is true.  We show
that for any $n>0$ the gap $g=2n$ occurs infinitely often in the stages of Eratosthenes sieve,
and we show that the ratio of occurrences of $g=2n$ to $g=2$, which ratio we denote
by $\Rat{1}{2n}$, asymptotically approaches the ratio implicit in Hardy and Littlewood's
Conjecture B \cite{HL}:
$$ \Rat{1}{2n}(\infty) = \prod_{q > 2, \gap q|n} \frac{q-1}{q-2}.$$

To accomplish this, we need to generalize the work in \cite{FBHgaps} and \cite{HRsmall}.
In those papers, we studied the cycle of gaps in each stage of Eratosthenes sieve,
denoting the corresponding cycle of gaps $\pgap(\primeprod{p_k})$.
Here we study the cycle of gaps $\pgap(N)$ among the generators in $\Z \bmod N$ for any $N$.

A {\em constellation among primes} \cite{Riesel} is a sequence of consecutive gaps
between prime numbers.  Let $s=c_1 c_2 \cdots c_j$ be a sequence of $j$
numbers.  Then $s$ is a constellation among primes if there exists a sequence of
$j+1$ consecutive prime numbers $p_{i_0} p_{i_0+1} \cdots p_{i_0+j}$ such
that for each $i=1,\ldots,j$, we have the gap $p_{i_0+i}-p_{i_0+i-1}=c_i$.  
In Eratosthenes sieve,
$s$ is a constellation if for some $p_k$ and some $i_0$ and all $i=1,\ldots,j$,
$c_i=g_{i_0+i}$ in $\pgap(\primeprod{p_k})$.  

For a constellation $s$, the {\em length} of $s$ is the number of gaps in $s$,
denoted $|s|$.
A {\em driving term} 
for a gap $g$ in $\pgap(\primeprod{p})$ is a constellation
whose gaps sum to $g$.  A driving term of length $1$ is the gap itself.

The power of the recursion on the cycle of gaps is seen in the following theorem,
which enables us to calculate the number of occurrences of a constellation $s$ through
successive stages of Eratosthenes sieve.

\begin{theorem}\label{CountThm}
(from \cite{FBHgaps,HRsmall})
Given a gap $g=2n$, let $n_{g,j}(p)$ be the number of driving terms for $g$ in 
$\pgap(\primeprod{p})$ of length $j$.
For every prime $p_k$ such that $g < 2p_{k+1}$,
\begin{equation}\label{NgjEq}
n_{g,j}(p_{k+1}) = (p_{k+1}-j-1)\cdot n_{g,j}(p_k) + j \cdot n_{g,j+1}(p_k).
\end{equation}
\end{theorem}

The challenge in applying this approach to Polignac's conjecture is that the condition
$g < 2p_{k+1}$ in Theorem~\ref{CountThm}
requires us to go far into the stages of Eratosthenes sieve,
before we can get exact counts for the driving terms of $g$ of each length $j$,
the $ n_{g,j}(p_k)$ in Equation~\ref{NgjEq}.

\begin{figure}[t]
\centering
\includegraphics[width=5in]{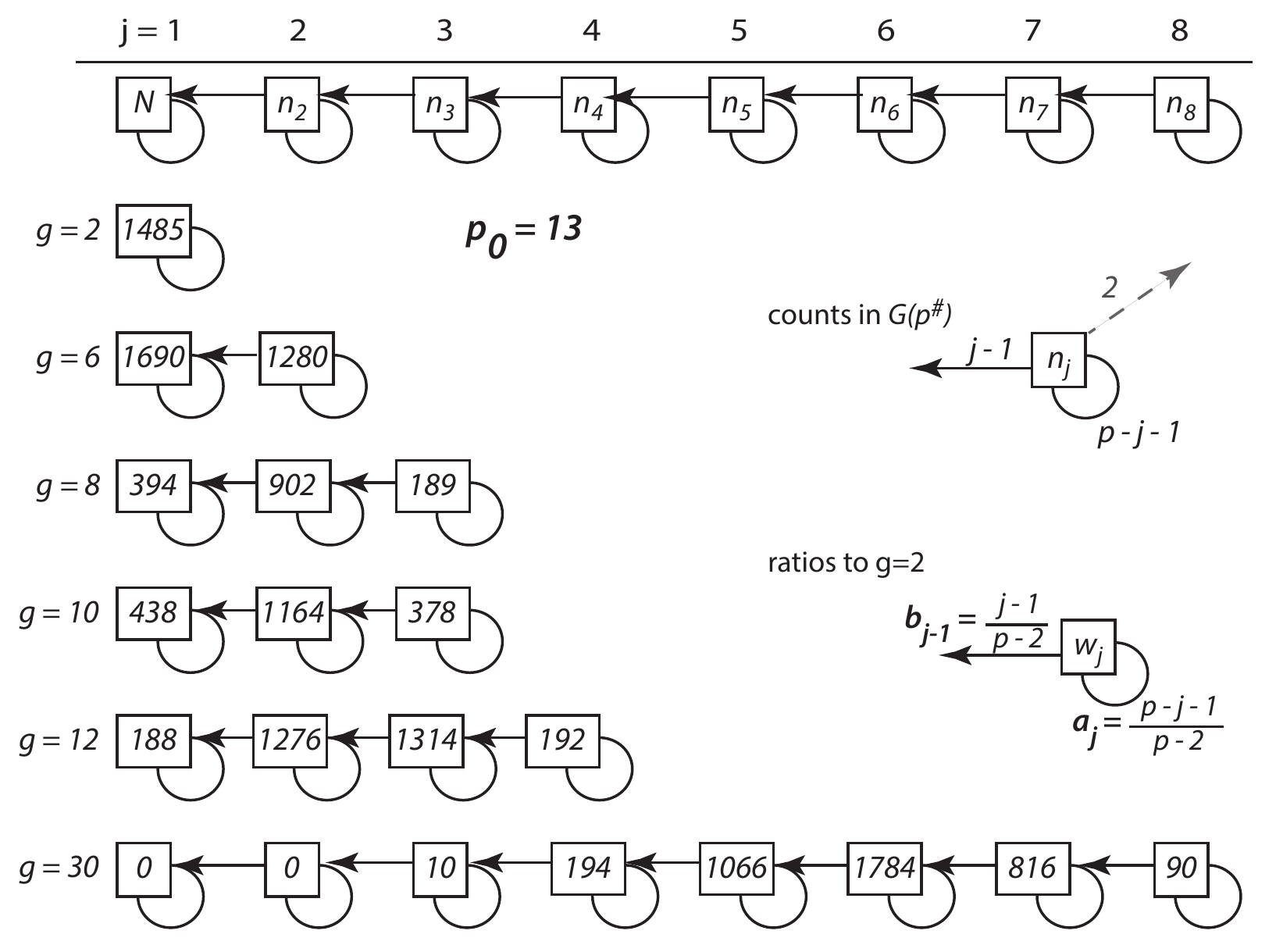}
\caption{\label{SystemFig} This figure illustrates the dynamic system of
Theorem~\ref{CountThm} through stages of the recursion for
$\pgap(\primeprod{p})$, using just the counts of gaps and their driving terms.
The coefficients of the system at each stage of the recursion
are independent of the specific gap and its driving terms.  
Below the diagram for the system, we record the initial conditions
for a set of gaps at $p_0=13$.  From this information we can derive the recursive count
for primes $q > p_0$.  Since the raw counts are superexponential, we take
the ratio of the count for each constellation to the count for $g=2$. }
\end{figure}

Following  \cite{HRsmall}, once we have the exact counts $n_{g,j}(p_k)$ 
for a prime $p_k$ such that $g < 2p_{k+1}$, 
we can set up a dynamic system representing the counts
through all subsequent stages of Eratosthenes sieve.  Since the gap $g=2$ has no driving
terms other than the gap itself, 
we take the ratios of the other gaps to the number of $2$'s at each stage of the sieve.
So instead of the raw counts of Equation \ref{NgjEq}, we use the ratios
$\Rat{j}{g}(p) = n_{g,j}(p) / n_{2,1}(p) $:
\begin{equation*}
\Rat{j}{g}(p_{k+1})  =  \frac{p_{k+1}-j-1}{p_{k+1}-2} \Rat{j}{g}(p_k)
     + \frac{j}{p_{k+1}-2} \Rat{j+1}{g}(p_k).
\end{equation*}

If we assemble the $\Rat{j}{g}(p_k)$ into a vector, we get a matrix equation
with a remarkably simple eigenstructure, from which we derive the following corollary
to Theorem~\ref{CountThm}.

\begin{corollary}\label{RatCor} 
(from \cite{HRsmall}) 
For any gap $g$ with initial ratios $\Rat{j}{g}(p_k)$, such that $g < 2p_{k+1}$,
the ratio of occurrences of this gap $g$
to occurrences of the gap $2$ in $\pgap(\primeprod{p})$ as $p \fto \infty$ converges to
the sum of these initial ratios across all the driving terms for this gap:
\begin{equation}\label{RatEq}
\Rat{1}{g}(\infty) =  \sum_j  \Rat{j}{g}(p_k).
\end{equation}
\end{corollary}

\subsection{Approach toward Polignac's conjecture.}
To establish an equivalent of Polignac's conjecture for Eratosthenes sieve, we show that for every
$n > 1$ the gap $g=2n$ does occur at some stage of the sieve and that as $p \fto \infty$
$$ \Rat{1}{g}(\infty) = \prod_{q>2, \; q|n} \frac{q-1}{q-2}.$$

To obtain this result, we first consider $\Z \bmod Q$ and its cycle of gaps $\pgap(Q)$,
in which $Q$ is the product of the prime divisors of $2n$.
We then bring this back into Eratosthenes sieve by filling in the primes
missing from $Q$ to obtain a primorial $\primeprod{p}$.

Once we are working with $\pgap(\primeprod{p})$, we are poised to apply
Theorem~\ref{CountThm} and Corollary~\ref{RatCor}.  However, the condition
$g < 2p_{k+1}$ could still require us to work with later stages of the sieve.
We are able to show that the conclusion of Corollary~\ref{RatCor} holds
under the construction we are using.

\section{The cycle of gaps among the generators in $\Z \bmod N$}
Let $\pgap(N)$ denote the cycle of gaps among the generators in $\Z \bmod N$, with
the first gap being that between $1$ and the next generator.  There are
$\phi(N)$ gaps in $\pgap(N)$ that sum to $N$.

There is a one-to-one correspondence between generators of $\Z \bmod N$ and the gaps in 
$\pgap(N)$.  Let
$$\pgap(N) = g_1 \; g_2 \; \ldots g_{\phi(N)}.$$
Then for $k < \phi(N)$, $g_k$ corresponds to the generator $\gamma = 1+\sum_{j=1}^{k} g_j$, and since
$\sum_{j=1}^{\phi(N)} = N$, the generator $1$ corresponds to $g_{\phi(N)}$.  Moreover, since
$1$ and $N-1$ are always generators, $g_{\phi(N)}=2$.  For any generator $\gamma$, $N-\gamma$ 
is also a generator, which implies that except for the final $2$, $\pgap(N)$ is symmetric.

In our previous work \cite{FBHgaps, HRest, HRsmall}, we focused on Eratosthenes sieve,
in which $N=\primeprod{p}$, the primorials.  For reference 
we provide a few base cases, since we will develop $\pgap(N)$ by building the cycle of gaps
via one prime factor of $N$ at a time.

\begin{remark}
\begin{enumerate}
\item For any prime number $p$, 
$$\pgap(p) = \underbrace{1\cdots1}_{p-2} 2$$
\item $\pgap(\primeprod{5}) = \pgap(30) = 64242462$.
\end{enumerate}
\end{remark}

As a convention, we write the cycles with the first gap being from $1$ to the next generator.  For
a prime $p$, every number is a generator in $\Z \bmod p$ except $p$ itself.
The last gap of $2$ is the gap from $p-1$ past $p$ (or $0$) around to $1$.
With $\pgap(p)$ as a starting point, we can build $\pgap(N)$ for any $N$ by introducing 
one prime factor at a time.

\begin{lemma}\label{LemmaCycle}
Given $\pgap(N)$, for a prime $q$ we construct $\pgap(qN)$ as follows:
\begin{enumerate}
\item[a)] if $q | N$, then we concatenate $q$ copies of $N$,
$$ \pgap(qN) = \underbrace{\pgap(N) \cdots \pgap(N)}_{q \gap {\rm copies}}$$
\item[b)] if $q \not| N$, then we build $\pgap(qN)$ in three steps:
\begin{enumerate}
\item[R1] Concatenate $q$ copies of $\pgap(N)$;
\item[R2] Close at $q$;
\item[R3] Close as indicated by the element-wise product $q * \pgap(N)$.
\end{enumerate}
\end{enumerate}
\end{lemma}

\begin{proof}
A number $\gamma$ in $\Z \bmod N$ is a generator iff $\gcd(\gamma,N)=1$. 
\begin{itemize}
\item[a)] 
Assume $q|N$.  Since $\gcd(\gamma,N)=1$, we know that $q\not| \gamma$.
For $j=0,1,\ldots,q-1$, we have 
$$\gcd(\gamma+jN, qN)= \gcd(\gamma,qN) = \gcd(\gamma,N)=1.$$
Thus $\gcd(\gamma,N)=1$ iff $\gcd(\gamma+jN,qN)=1$, and so the generators of $\Z \bmod qN$ have
the form $\gamma+jN$, and the gaps take the indicated form.
\item[b)] 
If $q \not| N$ then we first create a set of candidate generators for $\Z \bmod qN$, by
considering the set 
$$\set{\gamma+jN \st \gcd(\gamma,N)=1, \gap j=0,\ldots, q-1}.$$
For gaps, this is the equivalent of step R1, concatenating $q$ copies of $\pgap(N)$.
The only prime divisor we have not accounted for is $q$; if $\gcd(\gamma+jN,q)=1$, then this
candidate $\gamma+jN$ is a generator of $\Z \bmod qN$.  So we have to remove $q$ and its
multiples from among the candidates.

When we remove a multiple of $q$ as a candidate, we add together the gaps on each side of this multiple.
We call this {\em closing the gaps} at this multiple.

We first close the gaps at $q$ itself.  We index the gaps in the $q$ concatenated copies of
$\pgap(N)$:
$$ g_1 g_2 \ldots g_{\phi(N)} \ldots g_{q\phi(N)}.$$
Recalling that the first gap $g_1$ is the gap between the generator $1$ and the next smallest
generator in $\Z \bmod N$, the candidate generators are the 
running totals $\gamma_j = 1+\sum_{i=1}^{j-1} g_i$.  We take the $j$ for which $\gamma_j = q$, 
and removing $q$ from the list of candidate generators corresponds to replacing the gaps
$g_{j-1}$ and $g_j$ with the sum $g_{j-1}+g_j$.  This completes step R2 in the construction.

To remove the remaining multiples of $q$ from among the candidate generators, we note
that any multiples of $q$ that share a prime factor with $N$ have already been removed.
We need only consider multiples of $q$ that are relatively prime to $N$; that is, we only need
to remove $q\gamma_j$ for each generator $\gamma_j$ of $\Z \bmod N$ by closing the
corresponding gaps.

We can perform these closures by working directly with the cycle of gaps $\pgap(N)$.  Since
$q\gamma_{i+1}-q\gamma_{i} = qg_{i}$, we can go from one closure to the next by tallying the
running sum from the current closure until that running sum equals $qg_{i}$.
Technically, we create a series of indices beginning with $i_0=j$ such that $\gamma_j=q$,
and thereafter $i_k=j$ for which $\gamma_j-\gamma_{i_{k-1}}=q\cdot g_k$.  To cover the cycle
of gaps under construction, which consists initially of $q$ copies of $\pgap(N)$, $k$ runs
only from $0$ to $\phi(N)$.  We note that the last interval wraps around the end of the cycle
and back to $i_0$:  $i_{\phi(N)}=i_0$.
\end{itemize}
\end{proof}

\begin{theorem}\label{ThmqN}
In step R3 of Lemma~\ref{LemmaCycle}, each possible closure in $\pgap(N)$ occurs exactly once
in constructing $\pgap(qN)$.
\end{theorem}

\begin{proof}
Consider each gap $g$ in $\pgap(N)$.  Since $q \not| N$, $N \bmod q \neq 0$.
Under step R1 of the construction, $g$ has $q$ images.  Let the generator corresponding
to $g$ be $\gamma$.  Then the generators corresponding to the images of $g$ under step
R1 is the set:
$$\left\{ \gamma+jN \st j=0,\ldots,q-1\right\}.$$
Since $N \bmod q \neq 0$, there is exactly one $j$ for which $(\gamma+jN) \bmod q = 0$.
For this gap $g$, a closure in R2 and R3 occurs once and only once, at the image corresponding
to the indicated value of $j$.
\end{proof}

\begin{corollary}\label{qNCor}
Let $g$ be a gap.  If for the prime $q$, $q \not| g$, then 
$$ \sum \Rat{j}{g}(qN) = \sum \Rat{j}{g}(N).$$
\end{corollary}

\begin{proof}
Consider a driving term $s$ for $g$, of length $j$ in $\pgap(N)$.  
In constructing $\pgap(qN)$, we initially create $q$ copies of $s$.

If $q | N$, then the construction is complete.  For each driving term for $g$
in $\pgap(N)$ we have $q$ copies, and so
$ n_{g,j}(qN) = q  \cdot n_{g,j}(N).$  
Similarly $n_{2,1}(qN) = q \cdot n_{2,1}(N)$, and
$\Rat{j}{g}(qN) = \Rat{j}{g}(N).$  We have equality for each $j$ and so the result about the sum
is immediate.

If $q \not| N$, then in step R1 we create $q$ copies of $s$.
In steps R2 and R3, each of the possible closures in $s$ occurs once, distributed among the
$q$ copies of $s$.  The $j-1$ closures interior to $s$ don't change the sum, and the result
is still a driving term for $g$.  Only the two exterior closures, one at each end of $s$, change
the sum and thereby remove the copy from being a driving term for $g$.  Since $q \not| g$, 
these two exterior closures occur in separate copies of $s$.

If the condition $g < 2p_{k+1}$ applies, then each of the closures occur in a separate copy of 
$s$, and we can use the full dynamic system of Theorem~\ref{CountThm}.
For the current result we do not know that the closures necessarily occur in distinct copies of 
$s$, and so we can't be certain of the lengths of the resulting constellations.

However, we do know that of the $q$ copies of $s$, two are eliminated as driving terms and
$q-2$ remain as driving terms of various lengths.  
$$\sum_j n_{g,j}(qN) = (q-2) \sum_j n_{g,j}(N).$$
Since $n_{2,1}(qN) = (q-2)n_{2,1}(N)$, the ratios are preserved
$$\sum_j \Rat{j}{g}(qN) = \sum_j \Rat{j}{g}(N).$$
\end{proof}

\begin{corollary}\label{InfCor}
Let $g=2n$ be a gap, and let $\bar{q}$ be the largest prime factor of $g$.  Then
$$ \Rat{1}{g}(\infty) = \sum \Rat{j}{g}(\primeprod{\bar{q}}).$$
\end{corollary}

\begin{proof}
For Eratosthenes sieve, by Corollary~\ref{RatCor} and the preceding corollary, 
for all primes $p > \bar{q}$, 
$$ \sum \Rat{j}{g}(\primeprod{p}) = \sum \Rat{j}{g}(\primeprod{\bar{q}}),$$
and we have our result.
\end{proof}

\section{Polignac's conjecture for Eratosthenes sieve}
We establish an equivalent of Polignac's conjecture for Eratosthenes sieve.

\begin{theorem}\label{PolThm}
For every $n>0$, the gap $g=2n$ occurs infinitely often in Eratosthenes sieve, and the
ratio of the number of occurrences of $g=2n$ to the number of $2$'s converges asymptotically
to 
$$ \Rat{1}{2n}(\infty) = \prod_{q>2, \; q|n} \frac{q-1}{q-2}.
$$
\end{theorem}

We establish this result in two steps.  
First we find a stage of Eratosthenes sieve in which
the gap $g=2n$ has driving terms.  
Once we can enumerate the driving terms for $g$ 
in this initial stage of Eratosthenes sieve,
we can establish the asymptotic ratio of gaps $g=2n$ to the gaps $g=2$ as the sieve continues.

\begin{lemma}\label{QLem}
Let $g=2n$ be given.  Let $Q$ be the product of the primes dividing $2n$, including $2$,
$$Q = \prod_{q | 2n} q, \gap {\rm and} \gap n_1 = 2n / Q.$$
Finally, let $\bar{q}$ be the largest prime factor in $Q$.

Then in $\pgap(\primeprod{\bar{q}})$ the gap $g$ has driving terms, the total number 
of which satisfies
$$ \sum_j n_{g,j}(\primeprod{\bar{q}}) = \phi(Q) \cdot \prod_{p < \bar{q}, \; p\; \nmid \; Q} (p-2).$$
\end{lemma}

\begin{proof}
By Lemma~\ref{LemmaCycle} the cycle of gaps $\pgap(2n)$ consists of $n_1$ concatenated copies
of $\pgap(Q)$.  In $\pgap(Q)$, there are $\phi(Q)$ driving terms for the gap $g=2n$.  To see this, start
at any gap in $\pgap(Q)$ and proceed through the cycle $n_1$ times.  The length of each of these
driving terms is initially $n_1 \cdot \phi(Q)$.

We now want to bring this back into Eratosthenes sieve.  

Let $Q_0=Q$, and let $p_1, \ldots, p_k$ be the prime factors of $\primeprod{\bar{q}} / Q$.
For $i=1,\ldots, k$, let $Q_i = p_i \cdot Q_{i-1}$.  In forming $\pgap(Q_i)$ from $\pgap(Q_{i-1})$,
we apply Corollary~\ref{qNCor}.
Although we don't have
enough information about the lengths of the driving terms to apply the dynamic system 
of Theorem~\ref{CountThm} for each length $j$, we do know that 
\begin{equation*}
 \sum_{j=1}^{J} n_{2n,j}(Q_i) =  (p_i - 2) \cdot \sum_{j=1}^{J} n_{2n,j}(Q_{i-1})
 \end{equation*}
Thus at $p_k$ we have
\begin{eqnarray*}
 \sum_{j=1}^{J} n_{2n,j}(Q_k) &=&  (p_k - 2) \cdot \sum_{j=1}^{J} n_{2n,j}(Q_{k-1}) \\
  & = & \left( \prod_{i=1}^k (p_i-2) \right) \sum_{j=1}^J n_{2n,j}Q_0 
= \left( \prod_{i=1}^k (p_i-2) \right) \phi(Q) 
 \end{eqnarray*}
\end{proof}

\begin{proof} {\bf of Theorem~\ref{PolThm}.}
Let $g=2n$ be given.  Let $Q$ be the product of the prime factors dividing $g$ and
let $\bar{q}$ be the largest prime factor of $g$.  By Lemma~\ref{QLem} we know that
in $\pgap(\primeprod{\bar{q}})$ there occur driving terms for $g$ if not the gap $g$ itself.
We know the total number of these driving terms is
$$ \sum_j n_{g,j}(\primeprod{\bar{q}})  
 = \phi(Q) \cdot \prod_{p < \bar{q}, \; p\; \nmid \; Q} (p-2).$$
The number of gaps $2$ in $\pgap(\primeprod{q})$ is
 $n_{2,1}(\primeprod{q}) = \prod_{2<p\le q} (p-2).$
 So for the ratios we have
\begin{eqnarray*}
\sum_j \Rat{j}{g}(\primeprod{\bar{q}}) & = &
 \sum_j n_{g,j}(\primeprod{\bar{q}}) / n_{2,1}(\primeprod{\bar{q}}) \\
 &=& \phi(Q) / \prod_{p | Q, \; p > 2} (p-2)  = \prod_{p | Q, \; p > 2} \frac{(p-1)}{(p-2)}.
 \end{eqnarray*}
 By Corollary~\ref{qNCor} and Corollary~\ref{InfCor}, we have the result
 $$ \Rat{1}{2n}(\infty) = \prod_{p | 2n, \; p > 2} \left( \frac{p-1}{p-2}\right).$$
\end{proof}

This establishes a strong analogue of Polignac's conjecture for Eratosthenes sieve.
Not only do all even numbers appear as gaps in later stages of the sieve, but they
do so in proportions that converge to specific ratios.  
Using the gap $g=2$ as the reference point since it
has no driving terms other than the gap itself, 
the gaps for other even numbers appear in ratios to $g=2$
implicit in the work of Hardy and Littlewood \cite{HL}.

\section{Data \& Observations}

To anchor the above results in data, we exhibit a few tables of data extracted from stages
of Eratosthenes sieve.  In each table, the rows are indexed by the size of the gap, and the
columns are indexed by the length of the driving terms.  So if the table is for
the cycle of gaps $\pgap(\primeprod{p})$, the $ij^{\rm th}$ entry is the
number of driving terms for the gap $g_i$ of length $j$ in $\pgap(\primeprod{p})$.

We have calculated the tables for $\pgap(\primeprod{37})$, for gaps from $2$ to $3528$ and for
lengths $j$ of driving terms from $1$ to $500$.  
These tables are quite large.  We present samples from the larger tables, which are posted
on www.primegaps.com.

Our first table shows the table of nonzero entries for $\pgap(\primeprod{13})$, for gaps $g=2,\ldots,32$.
This is the range of gaps for which the condition $g < 2p_{k+1}$ holds and thus for which
the full dynamic system of Theorem~\ref{CountThm} applies.

\renewcommand\arraystretch{0.8}
\begin{center}
\begin{tabular}{|r|rrrrrrrrr|} \hline
gap & \multicolumn{9}{c|}{ $n_{g,j}(13)$: driving terms of length $j$ in $\pgap(\primeprod{13})$ }\\ [1 ex]
$g$ & $\lil j=1$ & $\lil 2$ & $\lil 3$ & $\lil 4$ & $\lil 5$ & $\lil 6$ & $\lil 7$ & $\lil 8$ & $\lil 9$ \\ \hline 
 $\lil 2, \; 4$ & $\lil 1485$ & & & & & & & & \\
 $\lil 6$ & $\lil 1690$ & $\lil 1280$ & & & & & & & \\
 $\lil 8$ & $\lil 394$ & $\lil 902$ & $\lil 189$ & & & & & & \\
 $\lil 10$ & $\lil 438$ & $\lil 1164$ & $\lil 378$ & & & & & & \\
 $\lil 12$ & $\lil 188$ & $\lil 1276$ & $\lil 1314$ & $\lil 192$ & & & & & \\
 $\lil 14$ & $\lil 58$ & $\lil 536$ & $\lil 900$ & $\lil 288$ & & & & & \\
 $\lil 16$ & $\lil 12$ & $\lil 252$ & $\lil 750$ & $\lil 436$ & $\lil 35$ & & & &  \\
 $\lil 18$ & $\lil 8$ & $\lil 256$ & $\lil 1224$ & $\lil 1272$ & $\lil 210$ & & & & \\
 $\lil 20$ &  & $\lil 24$ & $\lil 348$ & $\lil 960$ & $\lil 600$ & $\lil 48$ & & & \\
 $\lil 22$ & $\lil 2$ & $\lil 48$ & $\lil 312$ & $\lil 784$ & $\lil 504$ & & & & \\
 $\lil 24$ &  & $\lil 20$ & $\lil 258$ & $\lil 928$ & $\lil 1260$ & $\lil 504$ & & & \\
 $\lil 26$ & & $\lil 2$ & $\lil 40$ & $\lil 322$ & $\lil 724$ & $\lil 448$ & $\lil 84$ & & \\
 $\lil 28$ & &  & $\lil 36$ & $\lil 344$ & $\lil 794$ & $\lil 528$ & $\lil 80$ & & \\
 $\lil 30$ &  &  & $\lil 10$ & $\lil 194$ & $\lil 1066$ & $\lil 1784$ & $\lil 816$ & $\lil 90$ & \\
 $\lil 32$ &  &  &  & $\lil 12$ & $\lil 200$ & $\lil 558$ & $\lil 523$ & $\lil 172$ & $\lil 20$ \\ \hline
 \end{tabular}
 \end{center}
\renewcommand\arraystretch{1}

These results may lend some insight into the Jacobsthal function \cite{Ejac}.  
The Jacobsthal function ${\mathbf g}(N)$ is defined as
the least integer such that for any ${\mathbf g}(N)$ consecutive integers there is at least one which is
relatively prime to $N$.  We observe that 
this is equivalent to defining ${\mathbf g}(N)$ to be the maximum gap in $\pgap(N)$, and by Lemma~\ref{LemmaCycle} ${\mathbf g}(N)={\mathbf g}(Q)$, 
in which $Q$ is the product of the prime factors of $N$.
From Lemma~\ref{LemmaCycle} and Theorem~\ref{ThmqN}, letting ${\bar q}$ be the 
maximum prime in $Q$, we know that ${\mathbf g}(Q) \le {\mathbf g}(\primeprod{\bar q})$.

From our tabulated data, it appears that the
maximum gap that actually occurs in $\pgap(\primeprod{p_k})$ is roughly $2p_{k-1}$.  We
know from previous work \cite{FBHgaps} that the gap $g=2p_{k-1}$ always occurs
in $\pgap(\primeprod{p_k})$.  Although this gap is sometimes exceeded as the maximum gap,
the tables suggest that this value is often the maximum gap.

\begin{center}
\begin{tabular}{|rr|rr|rr|rr|}
\hline
\multicolumn{8}{|c|}{Maximum gap size occurring in $\pgap(\primeprod{p})$} \\
$p$ & $\max g$ & $p$ & $\max g$ & $p$ & $\max g$ & $p$ & $\max g$ \\ \hline
$3$ & $4$ & $11$ & $14$ & $19$ & $34$ & $31$ & $58$ \\
$5$ & $6$ & $13$ & $22$ & $23$ & $40$ & $37$ & $66$ \\
$7$ & $10$ & $17$ & $26$ &  $29$ & $46$ & $41$ &  $74$ \\ \hline
\end{tabular}
\end{center}

In the next table we exhibit the part of the table for $\pgap(\primeprod{31})$ at which the
driving terms through length $9$ are running out.  In this part of the table we observe
interesting patterns for the maximum gap associated with driving terms of a given length.
The driving terms of length $4$ have sums up to $90$ but none of sums $82$, $86$, or $88$.
Interestingly, although the gap $128$ is a power of $2$, in $\pgap(\primeprod{31})$ its 
driving terms span the lengths from $11$ to $27$; yet the gaps $g=126$ and $g=132$
already have driving terms of length $9$.

In each stage of Eratosthenes sieve, some copies of the driving terms of length $j$ will have
at least one interior closure, resulting in shorter driving terms at the next stage.  For this part of the table, 
$g \ge 2p_{k+1}$ and so more than one closure could occur within a single copy of a driving
term.  Letting $31=p_1$, we therefore know that a gap $g=2n$ will occur as a gap in 
$\pgap(\primeprod{p_k})$ for $k \le \min j$, the length of the shortest driving term for $g$
in $\pgap(\primeprod{31})$.

\renewcommand\arraystretch{0.8}
\begin{center}
\begin{tabular}{|r|rrrrrrrrr|rr|} \hline
gap & \multicolumn{9}{c|}{ $n_{g,j}(31)$: driving terms of length $j$ in $\pgap(\primeprod{31})$ }
 & & \\ [1 ex]
$g$ & $\lil j=1$ & $\lil 2$ & $\lil 3$ & $\lil 4$ & $\lil 5$ & $\lil 6$ & $\lil 7$ & $\lil 8$ & $\lil 9$
 & $\lil \sum \Rat{j}{g}$ & $\lil \Rat{1}{g}(\infty)$ \\ \hline 
$\lil 74$ &  &  & $\lil 1$ & $\lil 1206$ & $\lil 70194$ & $\lil 1550662$
 & $\lil 17523160$ & $\lil 113497678$ & $\lil 445136490$ & $\lil 1$ & $\lil 1.02857$ \\
$\lil 76$ &  & &  & $\lil 602$ & $\lil 32194$ & $\lil 765488$
  & $\lil 9470176$ & $\lil 68041280$ & $\lil 302507798$ & $\lil 1.0588$ & $\lil 1.0588$\\
$\lil 78$ &  & & & $\lil 292$ & $\lil 26060$ & $\lil 826426$
   & $\lil 12166908$ & $\lil 99284264$ & $\lil 489040926$ & $\lil 2.1818$  & $\lil 2.1818$ \\
$\lil 80$ &  &  &  & $\lil 2$ & $\lil 2876$ & $\lil 139926$ & $\lil 2656274$
    & $\lil 26634332$ & $\lil 159280176$  & $\lil 1.3333$  & $\lil 1.3333$\\
$\lil 82$ &  &  &  &  & $\lil 747$ & $\lil 46878$ & $\lil 1066848$
     & $\lil 12378176$ & $\lil 83484438$ & $\lil 1$ & $\lil 1.0256$ \\
$\lil 84$ & &  & & $\lil 2$ & $\lil 1012$ & $\lil 58216$ & $\lil 1485176$
      & $\lil 18772184$ & $\lil 135450260$ & $\lil 2.4$ & $\lil 2.4$ \\
$\lil 86$ & &  & & & $\lil 74$ & $\lil 4726$ & $\lil 147779$
       & $\lil 2453256$ & $\lil 23265268$ & $\lil 1$ & $\lil 1.0244$ \\
$\lil 88$ & & &  & & $\lil 2$ & $\lil 2190$ & $\lil 107182$
        & $\lil 2025910$ & $\lil 20603366$ & $\lil 1.1111$ & $\lil 1.1111$ \\
$\lil 90$ &  &  & & $\lil 8$ & $\lil 300$ & $\lil 9360$ & $\lil 195708$
         & $\lil 2829548$ & $\lil 26983182$ & $\lil 2.6667$ & $\lil 2.6667$ \\
$\lil 92$ & &  & & & $\lil 20$ & $\lil 860$ & $\lil 26854$
          & $\lil 488854$ & $\lil 5364068$ & $\lil 1.0476$ & $\lil 1.0476$ \\
$\lil 94$ &  &  & & & $\lil 16$ & $\lil 740$ & $\lil 19740$
           & $\lil 333162$ & $\lil 3684805$ & $\lil 1$ & $\lil 1.0222$ \\
$\lil 96$ & & & &  & $\lil 4$ & $\lil 242$ & $\lil 9636$
            & $\lil 249610$ & $\lil 3693782$ & $\lil 2$ & $\lil 2$ \\
$\lil 98$ &  & &  &  & & $\lil 28$ & $\lil 1482$ & $\lil 52328$
             & $\lil 968210$ & $\lil 1.2$ & $\lil 1.2$ \\
$\lil 100$ &  & & & & & $\lil 8$ & $\lil 672$
              & $\lil 26428$ & $\lil 567560$ & $\lil 1.3333$ & $\lil 1.3333$ \\
$\lil 102$ &  & &  & &  & & $\lil 78$ & $\lil 7042$
               & $\lil 249300$ & $\lil 2.133$ & $\lil 2.133$ \\
$\lil 104$ &  &  &  &  &  &  & $\lil 182$ & $\lil 6086$
                & $\lil 129016$ & $\lil 1.0909$ & $\lil 1.0909$ \\
$\lil 106$ &  & &  & & &  & $\lil 16$ & $\lil 1168$
                 & $\lil 37144$ & $\lil 1$ & $\lil 1.0196$ \\
$\lil 108$ & & &  &  &  &  & $\lil 8$ & $\lil 1244$
                  & $\lil 44334$ & $\lil 2$ & $\lil 2$ \\
$\lil 110$ &  & & & & & &  & $\lil 142$
                   & $\lil 7686$ & $\lil 1.4815$ & $\lil 1.4815$ \\
$\lil 112$ &  & &  &  &  & &  & $\lil 68$
                    & $\lil 5294$ & $\lil 1.2$ & $\lil 1.2$ \\
$\lil 114$ & &  & & & & &  & $\lil 22$
                     & $\lil 2388$ & $\lil 2.1176$ & $\lil 2.1176$ \\
$\lil 116$ & & &  & &  & & & $\lil 224$
                      & $\lil 4716$ & $\lil 1.0370$ & $\lil 1.0370$ \\
$\lil 118$ & & &  &  &  &  &  &  & $\lil 72$
                      & $\lil 1$ & $\lil 1.0175$ \\
$\lil 120$ & &  & &  &  &  &  &  & $\lil 1012$ 
                      & $\lil 2.6667$ & $\lil 2.6667$ \\
$\lil 122$ & &  & &  & &  &  &  & $\lil 70$
                      & $\lil 1$ & $\lil 1.0169$ \\
$\lil 124$ & &  & &  & &  &  &  & $\lil 28$
                      & $\lil 1.0345$ & $\lil 1.0345$ \\
$\lil 126$ & &  & &  & &  &  &  & $\lil 4$
                      & $\lil 2.4$ & $\lil 2.4$ \\
$\lil 128$ & &  & &  & &  &  &  & 
                      & $\lil 1$ & $\lil 1$ \\
$\lil 130$ & &  & &  & &  &  &  & 
                      & $\lil 1.4545$ & $\lil 1.4545$ \\
$\lil 132$ & &  & &  & &  &  &  & $\lil 2$
                      & $\lil 2.2222$ & $\lil 2.2222$ \\
\hline
 \end{tabular}
 \end{center}
\renewcommand\arraystretch{1}

From the tabled values for $\pgap(\primeprod{31})$, we see that the driving term of length $3$
for $g=74$ will advance into an actual gap in two more stages of the sieve.  Thus the maximum
gap in $\pgap(\primeprod{41})$ is at least $74$, and the maximum gap for
$\pgap(\primeprod{43})$ is at least $90$.

For $g=74, 82, 86, 94,106, 118, 122$, note that in the table for $\pgap(\primeprod{31})$
$$\sum_j \Rat{j}{g}(\primeprod{31}) \neq \Rat{1}{g}(\infty).$$
 Up through $\pgap(\primeprod{31})$
the ratio is $1$; but for each gap, we know that this ratio will jump
to equal $\Rat{1}{g}(\infty)$ in the respective $\pgap(\primeprod{\bar{q}})$.  
How does the ratio transition from $1$ to the
asymptotic value?  If we look further in the data for $\pgap(\primeprod{31})$, we see that for
the gap $g=222$, $\sum_j \Rat{j}{222}(\primeprod{31})=2$ but the asymptotic value is
$\Rat{1}{222}(\infty) = 72/35.$

These gaps $g=2n$ have maximum prime divisor $\bar{q}$ greater than the prime $p$ for
the current stage of the sieve $\pgap(\primeprod{p})$.  From Corollary~\ref{qNCor} and the
approach to proving Lemma~\ref{QLem}, we are able to establish the following.

\begin{corollary}
Let $g=2n$, and let $Q=q_1 q_2 \cdots q_k$ be the product of the distinct prime factors of $g$,
with $q_1 < q_2 < \cdots < q_k$.  Then for $\pgap(\primeprod{p})$,
$$
\sum_j \Rat{j}{g}(\primeprod{p}) = \prod_{2 < q_i \le p} \left(\frac{q_i-1}{q_i-2}\right).
$$
\end{corollary}

\begin{proof}
Let $p=q_j$ for one of the prime factors in $Q$.  By Corollary~\ref{qNCor} these are the only
values of $p$ at which the sum of the ratios $\sum_j \Rat{j}{g}(p)$ changes.

Let $Q_j = q_1 q_2 \cdot q_j$.
In $\pgap(\primeprod{q_j})$, $g$ behaves like a multiple of $Q_j$. 
As in the proof of Lemma~\ref{QLem}, in $\pgap(Q_j)$ each generator begins a driving term
of sum $2n$, consisting of $2n / Q_j$ complete cycles.  There are $\phi(Q_j)$ such driving
terms.  

We complete $\pgap(\primeprod{q_j})$ as before by introducing the missing prime factors.  
The other prime factors do not
divide $2n$, and so by Corollary~\ref{qNCor} the sum of the ratios is unchanged by these factors.
We have our result:
$$ \sum_j \Rat{j}{g}(\primeprod{q_j}) = \prod_{2 < q_i \le q_j} \left(\frac{q_i-1}{q_i-2}\right).
$$
\end{proof}

Once the gap $g=2n$ finally occurs in $\pgap(\primeprod{p})$, from the description 
of the dynamic system in \cite{HRsmall}, we know that the ratio $\Rat{1}{g}(\primeprod{p})$ 
converges to its asymptotic value
as quickly as 
$$a_2^k=  \prod_{q=p_1}^{p_k} \frac{q-3}{q-2}$$ 
converges to $0$.  This convergence is very slow; for
$p_1=17$ and $p_k \approx 3.01 \times 10^{15}$, $a_2^k$ is still around $0.079138$.

\section{Conclusion}
By identifying structure among the gaps in each stage of Eratosthenes sieve, we 
have been able, for a handful of conjectures about gaps between primes,
to resolve the equivalent conjectures for Eratosthenes sieve.  These results
provide evidence toward the original conjectures, to the extent that gaps
in stages of Eratosthenes sieve are indicative of gaps among primes themselves.

In \cite{FBHgaps} we established that across the stages of Eratosthenes sieve:
\begin{itemize}
\item {\em Spikes - liminf.}  In the $\ord{k}$ stage of the sieve, there are consecutive gaps
$g_{k1}$ and $g_{k2}$ such that as $k \fto \infty$, $\lim \inf g_{k2} / g_{k1} = 0$.
\item {\em Spikes - limsup.}  In the $\ord{k}$ stage of the sieve, there are consecutive gaps
$g_{k1}$ and $g_{k2}$ such that as $k \fto \infty$, $\lim \sup g_{k2} / g_{k1} = \infty$.
\item{\em Superlinear growth.} For any $n > 2$, there exists a stage $k_n$ of the sieve, such that
for all stages $k \ge k_n$, there exists a sequence of $n$ consecutive gaps 
$g_{k,i+1}, \ldots, g_{k,i+n}$
$$ g_{k,i+1} < g_{k,i+2} < \cdots < g_{k,i+n}.$$
\item{\em Superlinear decay.} For every $n > 2$, there exists a stage $k_n$ of the sieve, such that
for all stages $k \ge k_n$, there exists a sequence of $n$ consecutive gaps 
$g_{k,i+1}, \ldots, g_{k,i+n}$
$$ g_{k,i+1}  > g_{k,i+2} > \cdots > g_{k,i+n}.$$
\end{itemize}
These results provide examples that persist through all subsequent stages of Eratosthenes 
sieve and thereby provide
evidence to resolve conjectures by Erd{\"o}s and Tur{\' a}n \cite{ET}.

In this paper we have generalized the approach we have used in \cite{FBHgaps, HRest, HRsmall}
in order to establish for Eratosthenes sieve the analogue of Polignac's conjecture. 
We have shown that for every $n>0$, there is a stage $k_n$ such that for every
stage $k \ge k_n$ of the sieve there exist gaps of size $g=2n$.

Moreover, we have shown that the ratio of the number of gaps $g=2n$ to the number of gaps $2$
in the $\ord{k}$ stage of the sieve, which ratio is denoted $\Rat{1}{2n}(p_k)$, asymptotically 
approaches the ratio suggested by Hardy and Littlewood \cite{HL}:
$$ \Rat{1}{2n}(\infty) = \prod_{q>2, \; q|n} \frac{q-1}{q-2}.$$


\bibliographystyle{amsplain}

\providecommand{\bysame}{\leavevmode\hbox to3em{\hrulefill}\thinspace}
\providecommand{\MR}{\relax\ifhmode\unskip\space\fi MR }
\providecommand{\MRhref}[2]{%
  \href{http://www.ams.org/mathscinet-getitem?mr=#1}{#2}
}
\providecommand{\href}[2]{#2}

\end{document}